\documentclass[12pt]{article}

\usepackage{amsmath,amsfonts,amssymb,amsthm,color,mathrsfs}

\newcommand{\R}{\mathbb{R}}
\newcommand{\N}{\mathbb{N}}
\newcommand{\Z}{\mathbb{Z}}

\newcommand{\kk}[1]{{_q}{#1}}

\newcommand{\pr}{\mathbf{P}}
\newcommand{\exn}{\mathbf{E}}
\newcommand{\bz}{\mbox{\boldmath$z$}}
\newcommand{\by}{\mbox{\boldmath$y$}}
\newcommand{\bx}{\mbox{\boldmath$x$}}
\newcommand{\bw}{\mbox{\boldmath$w$}}
\newcommand{\bu}{\mbox{\boldmath$u$}}
\newcommand{\bv}{\mbox{\boldmath$v$}}
\newcommand{\bm}{\mbox{\boldmath$m$}}
\newcommand{\bn}{\mbox{\boldmath$n$}}
\newcommand{\bo}{\mbox{\boldmath$0$}}
\newcommand{\bzeta}{\mbox{\boldmath$\zeta$}}
\newcommand{\bxi}{\mbox{\boldmath$\xi$}}
\newcommand{\balpha}{\mbox{\boldmath$\alpha$}}

\newcommand{\sbz}{\mbox{\scriptsize\boldmath$z$}}
\newcommand{\tbz}{\mbox{\tiny\boldmath$z$}}

\newcommand{\sbx}{\mbox{\scriptsize\boldmath$x$}}
\newcommand{\sbv}{\mbox{\scriptsize\boldmath$v$}}
\newcommand{\sbm}{\mbox{\scriptsize\boldmath$m$}}
\newcommand{\sbn}{\mbox{\scriptsize\boldmath$n$}}
\newcommand{\sbal}{\mbox{\scriptsize\boldmath$\alpha$}}
\newcommand{\sbzeta}{\mbox{\scriptsize\boldmath$\zeta$}}

\newtheorem{thm}{Theorem}
\newtheorem{exam}{Example}
\newtheorem{rema}{Remark}

\begin{document}

\title{On stationary distributions of stochastic neural networks}

\author{K.~Borovkov,\footnote{Department of Mathematics and Statistics, The University of Melbourne, Parkville 3010, Australia; e-mail: {borovkov@unimelb.edu.au}.}
 \
 G.~Decrouez,\footnote{Department of Mathematics and Statistics, The University of Melbourne, Parkville 3010, Australia; e-mail:
 {dgg@unimelb.edu.au}.}
 \
 M.~Gilson\footnote{Department of Electrical and Electronic Engineering, The University of Melbourne, Parkville, Australia. Now at: Laboratory  for Neural Circuit Theory, RIKEN Brain Science Institute, Hirosawa 2-1, Wako-shi, Saitama 351-0198, Japan; e-mail: gilson@brain.riken.jp.}
 }

\date{}

\maketitle

\renewcommand{\thefootnote}{}
\footnotetext{Research supported by the ARC Discovery Grant DP120102398 and the ARC Centre of Excellence for Mathematics and Statistics of Complex Systems.}

\begin{abstract}
The paper deals with non-linear Poisson neuron network models with bounded memory dynamics, that can include both Hebbian learning mechanisms and refractory periods.  The state of a network is described by  the times elapsed since its neurons fired within the post-synaptic transfer kernel memory span, and the current strengths of synaptic connections, the state spaces of our models being hierarchies of finite-dimensional components. We establish ergodicity of the stochastic processes describing the behaviour of the networks and prove the existence of continuously differentiable stationary distribution densities (with respect to the Lebesgue measures of corresponding dimensionality) on the components of the state space and find upper bounds for them. For the density components, we derive a system  of differential equations that can be solved in a few simplest cases only. Approaches to approximate computation of the stationary density are discussed. One is to reduce the dimensionality of the problem by modifying the network so that each neuron cannot fire if the number of spikes it emitted within the post-synaptic transfer kernel memory span reaches a given threshold. We show that the stationary distribution of this `truncated' network converges to that of the unrestricted one as the threshold increases, and that the convergence is at a super-exponential rate. A complementary approach uses discrete Markov chain approximations to the network process. We derive linear systems for the stationary distributions of these Markov chains and prove that these distributions  converge weakly to the stationary laws for the original processes.
\end{abstract}

\small{\textit{\hspace{.2cm} Keywords:} Neural networks, non-linear Poisson neuron, Markov process, ergodicity, stationary distribution.}

\maketitle

\section{Introduction}

Neurons are  electrically excitable cells whose main function is to process and transmit  information. They connect to each other to form neural networks that constitute core components of the nervous system, and so building and studying mathematical models of such networks is of key interest. To justify the modelling approach used in this paper (as described in detail in Section~\ref{secMod}), we will briefly describe
the mechanism enabling neurons to communicate with each other.

The anatomy of a neuron involves three distinct parts with different electrical activity functions: dendrites that form a tree and contain post-synaptic
receptors (inputs), the cell body (soma) that integrates the input currents coming from the dendrites, and a long-limbed axon that terminates with pre-synaptic buttons (outputs). A typical feature of the neuronal electrical activity is the propagation of membrane depolarisation. The membrane of a resting neuron is polarised. Brief high-amplitude depolarisations that propagate from the soma along the axon are called action potentials (or spikes) and have a characteristic shape. When a spike reaches an axonal termination that ``connects'' to a post-synaptic neuron, neurotransmitters are released into the extra cellular space and
excite receptors on the post-synaptic neuron (usually on dendrites). This generates a local variation of the membrane potential in that neuron, which propagates towards the soma. The soma can be seen as a spatio-temporal integrator of these post-synaptic potentials (PSP) to generate an output spike. The soma potential often remains close to the resting value for a
few milliseconds after firing an action potential, which is referred to as the
refractory period (through which the neuron cannot fire again). These basic elements of the neuronal information processing actually depend upon many different mechanisms at the molecular level, such as ionic
concentrations, density of ion channels, axonal myelination, and types of
neurotransmitters (for a review, we refer the reader to~\cite{BeaEta07}).

There exists extensive literature on mathematical modelling of both individual neurons and neural networks. We refer the reader interested in neurophysiological principles of neuron and brain operation to~\cite{BeaEta07, NiEtal11} and  more advanced expositions in~\cite{She03, SheGri10} of the circuitry of the brain. A detailed (but accessible and rather non-technical) discussion of brain networks, covering structural, functional, and effective connectivity and their respective dynamics, is presented in~\cite{Spo10} (the book also contains extensive bibliography of the relevant research work in the complex network theory). A detailed overview of the computational modelling of nervous systems from the molecular and cellular level, including mathematical modelling of adaptation and learning, is given in~\cite{DavAbb05}. Monograph~\cite{Izh10} is a systematic study of the relationship of electrophysiology, nonlinear dynamics, and computational properties of neurons. One can also mention here~\cite{Koch, GerKis02} and refer to~\cite{Gilson09} for a recent review of the literature in the area.

As the duration of an action potential is relatively short (usually less than 1~ms for sodium-based action potentials), for modelling purposes spikes
are often considered to be instantaneous. Hence a neuron can be modelled using a point process whose intensity depends on the past activity of the neuron, its incoming synaptic stimulation or other mechanisms. The use of point processes, such as Hawkes processes, for modelling the spiking activity of neurons dates back to papers~\cite{Bril75, Bril88, Chor88} and made it possible to study analytically the neuronal response to various input stimulations (for a review, see~\cite{GerKis02}). Applications of particular neural non-linear point process models to real data can be found in~\cite{Panin04, Pill08, Pill11, Stev09, Truc05}.

In Section~\ref{secMod} we present descriptions of two network models we are dealing with in this paper. Both are ``assembled" of non-linear Poisson neurons that can be viewed as extended versions of self-exciting Hawkes point processes, the difference between the two models being that the former has constant strengths of synaptic connections between neurons, whereas in the latter, to model Hebbian learning, we allow the strengths to change depending on the order in which the connected neurons are firing. Instead of using the formalism of point processes (as e.g.\ in~\cite{BremMass96, BremM02}), we choose an alternative description in terms of multivariate Markov processes whose states represent networks' spiking histories, with state spaces being products of hierarchies of simplices. This approach proves to be rather convenient and allows one to demonstrate ergodicity of the network processes under rather general conditions (Section~\ref{secErg}) and, moreover, to study the stationary distributions thereof. We show  that the stationary distribution of the network has a smooth density with respect to a natural measure on the state space and give upper bounds for the components of that density on different components of the state space of the process.

In Section~\ref{appra} we discuss a way to reduce the dimensionality of the model and approximate its stationary distribution with more tractable objects. The approach is based on ``truncating" the original process by ``forbidding" neurons to fire once they have fired a given number $n$  of spikes recently (within the ``memory window" of the neuron). Thus modified process will still be Markovian and ergodic, its stationary distribution confined to a space of lower dimensionality and approximating that of the original process at a super-exponential rate in~$n$. In fact, such dynamics do make physical sense when the existence of the refractory period is taken into account, but one can further simplify the model by choosing an even lower threshold~$n$. 
 
Section~\ref{secComp} deals with the problem of computing the stationary distributions of the networks. We derive systems of differential equations for the stationary distributions (unfortunately, they seem to be tractable in the simplest cases only, that are discussed as examples). Moreover, we prove that the stationary distribution of our network process can be approximated by those of discrete Markov chains constructed as discretised (in both time and space) versions of the process. Computing the stationary distributions for the chains is more feasible, as it only requires solving systems of (a large number of) linear algebraic equations.

\section{Network dynamics and its description by Markov processes}
\label{secMod}

We consider a model neural network consisting of $N$ neurons and $M$ external sources. Both external  sources and neurons can fire spikes, which are assumed to be generated by a random mechanism. External sources are assumed to fire according to independent Poisson processes with constant rates $\hat \rho_k$, $k\in \{1,2,\ldots, M\}$ (all the quantities related to external sources will be labelled with hats, and all the processes in the paper will be assumed to be right-continuous), whereas a neuron's instantaneous firing rate  is determined by the value of the activation function of the so-called synaptic influx. For neuron $i$, the latter is the sum of all  PSPs generated by spikes arriving to the  $i$th neuron's synapses from external sources and other neurons, and also background activity.

More precisely, assuming that $\{\widehat{T}_{k,n}\}_{n\in\Z}$ are times at which external source $k\in\{1,\ldots, M\}$ fired, $\{T_{j,n}\}_{n\in\Z}$ are times at which neuron $j\in\{1,\ldots, N\}$ fired, and $\hat\epsilon_{ik}(t)$ and $\epsilon_{ij}(t)$ are post-synaptic response kernel functions describing the effects on neuron $i$ potential from accepting spikes through synapses connecting source $k$ to neuron $i$ and neuron $j$ to neuron~$i$, respectively, the total time $t$ synaptic influx for neuron~$i$ is given by
\[
J_i (t):= v_i +  \sum_{k, m}\widehat{W}_{ik}(\widehat{T}_{k,m})   \hat{\epsilon}_{ik}(t-\widehat{T}_{k,m})
 +\sum_{j, n}W_{ij}(T_{j,n})  \epsilon_{ij}(t-T_{j,n}) ,
\]
where $v_i=\rm{const}$ represents the background activity for neuron~$i$,
the synaptic weights $\widehat{W}_{ik}(t)$ and $W_{ij}(t)$ can be positive (excitatory synapse)
or negative (inhibitory synapse) and, to reflect brain plasticity  (e.g.\ to model Hebbian learning), they can depend on time as well. If, at time $t$, there is no synaptic connection of external source~$k$ to neuron~$i$, we simply have $\widehat{W}_{ik}(t)=0$, and likewise for network neurons' connections.

The kernels $\hat\epsilon_{ik}(t)\ge 0$ and $\epsilon_{ij}(t)\ge 0$ are assumed to vanish outside a compact interval: for any  $k\le M$ and $i,j\le N,$
\[
\hat\epsilon_{ik}(t)=\epsilon_{ij}(t)=0\quad\mbox{for}\quad  t\not\in [0,\Theta],
 \quad \Theta = \mbox{const}>0,
\]
which ensures causality and also means that the {\em direct\/} effect of  any given spike on a neuron completely disappears within a finite time $\Theta$ (for real life neurons, the order of magnitude of $\Theta$ is $10^2$~ms).

In a previous series of papers by one of the authors (\cite{Gilson09c, Gilson10}), the case where all kernels were identical to some function $\epsilon$, but incorporating individual synaptic delays $\hat{d}_{ik}$ and $d_{ij}$:
\[
\hat{\epsilon}_{ik}(t)=\epsilon(t-\hat{d}_{ik})\ \mbox{ and } \ \epsilon_{ij}(t)=\epsilon(t-d_{ij})
\]
was considered. The delays account for both the axonal propagation of action potential up to the synaptic site, and for the diffusion time of the neurotransmitters in the synaptic
cleft. In the present paper, we allow each synapse to have individual properties.

The effect of the synaptic influx on the behaviour of  neuron $i$ is expressed via an activation function $\varsigma_i (\cdot)$, which is assumed to be continuous non-decreasing (and usually ``S-shaped"), with
\[
0<\underline{\varsigma}_i \le \varsigma_i (x) \le \overline{\varsigma}_i
 \le \overline{\varsigma}:=\max_j \overline{\varsigma}_j <\infty, \quad x\in\R, \ i\le N.
\]
Namely, denoting by $T_i (t)$ the time of the last spike fired by neuron $i$ prior to time $t$, by $\mathscr{F}_t$ the $\sigma$-algebra generated by the evolution of our system up to time $t$ and setting $\Delta (t):= (t, t+\Delta)$ for $\Delta>0$, we have, as $ \Delta \to 0,$
\begin{equation}
\pr \bigl( \mbox{neuron $i$ fires during } \Delta (t)\, |\, \mathscr{F}_t\bigr)
 = \varsigma_i (J_i(t))  r (t- T_i (t))\Delta + o (\Delta),
 \label{prs}
\end{equation}
where we used a left-continuous function $r(\cdot)\in [0,1]$ to model  the existence of the so-called
absolute refractory period, i.e.\ the time period during which a just fired neuron is  unable to fire again. One can take  e.g.\
\begin{equation}
r(s):=\mathbf{1} (s\not\in (0,\delta_{AR}]), \quad \delta_{AR}=\mbox{const}>0,
 \label{refr}
\end{equation}
the indicator function of the complement of the interval $(0,\delta_{AR}]$ (for real life neurons, $\delta_{AR}$ is about 1~ms). Whatever the shape of $r$, we always assume that $r(s)=1$ for $s\ge \Theta$.

In addition to~\eqref{prs}, we assume that
\[
\pr \bigl( \mbox{more than one neuron fires during } \Delta (t)\, |\, \mathscr{F}_t\bigr)
 =   o (\Delta),
\]
which basically means that, given the past history $\mathscr{F}_t$, the instantaneous firing of different neurons is driven by independent random mechanisms.

Note that the widely studied classical Hawkes process (see \cite{Hawkes, Brem81, BremM01, BremM02, KGvH}) corresponds to the identity activation function $\varsigma_i$ in~\eqref{prs}, and that the positivity of $\varsigma_i$ means that  neurons can fire  spikes in the absence of any external stimulation.  The use of  non-linear  bounded functions $\varsigma_i$ is motivated by the experimentally observed saturation of the neuronal firing rate when
its excitation increase. Observe also that the temporal spread of the synaptic responses (modelled by $\epsilon_{ij}$) induces specific temporal correlations between the neuronal spike trains, which can be evaluated for the
case of linear activation functions $\varsigma_i$~\cite{Hawkes, Gilson09d, Gilson10}.

We will consider two types of models that differ in their assumptions concerning synaptic weights:

\smallskip

{\bf Model~I} assumes that all the synaptic weights are constant: $\widehat{W}_{ik}(t) \equiv \widehat{W}_{ik}=\mbox{const}$ and $W_{ij}(t)\equiv W_{ij}=\mbox{const}$ for any  $k\le M$ and $i,j\le N$.

 \smallskip

{\bf Model~II} assumes a Hebbian learning mechanism in the form of spike-timing
dependent plasticity (STDP) : if, within a short enough time interval, there are spikes
at both pre-synaptic and post-synaptic sides of a connection, this can change the
weight of the connection. The weight increases if the post-synaptic spike follows the
pre-synaptic one (reinforcement of the synapse), and decreases otherwise (depression of
the synapse).

For simplicity we assume that, for each of the connections, the synaptic weight can assume finitely many values: for a common finite $L$,
\begin{align*}
\widehat{W}_{ik}(t) & \in \widehat{G}_{ik}
  := \{\hat g_{ik} (1) \le \hat g_{ik} (2)\le \cdots  \le  \hat g_{ik} (L)\},
  \\
W_{ij}(t) & \in {G}_{ij}
 := \{  g_{ij} (1) \le   g_{ij} (2) \le\cdots \le   g_{ij} (L)\},
\end{align*}
and the following discrete approximation of the STDP mechanisms discussed e.g.\
in~\cite{Buetal07}.  For any $i,j\in\{1,\ldots, N\}$ and $m\in\{1,\ldots, L\}$, we have
a collection of points
\begin{align*}
- \infty & <
 u_{ij} (m,m+1) <  u_{ij} (m,m+2)   < \cdots < u_{ij} (m,L+1)=0\\
  &=  u_{ij} (m,1)
      <  u_{ij} (m,2)  < \cdots < u_{ij} (m,m) <\infty
\end{align*}
In real life situations, the length $\delta_{LW} :=\max\{u_{ij} (m,m), | u_{ij}
(m,m+1)|\}$  of the ``learning window" is about $10^2$~ms. We assume that $\delta_{LW}
<\Theta$.

Now suppose that, for a given time $t$, one has $W_{ij}(t-)= g_{ij}(m) $ and either
$t=T_i(t)$ or $t=T_j(t)$ (i.e.\ one of the neurons $i,j$  fired at time~$t$). Then we
put
\[
W_{ij}(t):= g_{ij}(d)\quad \mbox{if} \quad \left\{
 \begin{array}{l}
 t=T_i(t)\ \mbox{ and }\ T_j (t) -t \in (u_{ij} (m,d), u_{ij} (m,d+1)], \\
 t=T_j(t)\ \mbox{ and }\ t -T_i (t)  \in (u_{ij} (m,d), u_{ij} (m,d+1)],
 \end{array}
 \right.
\]
$d\in\{1,\ldots, L\}$, while otherwise the value of the weight remains unchanged. A
similar rule applies to the weights~$\widehat W_{ik}$.

Note that the above mechanism  allows one to   model the emergence of new synaptic connections as well. Altogether, our network models provide a certain degree of biologically realism together with a mathematical framework that allows a tractable analysis.

\smallskip

Having described the rules governing of our neural network, we will  now present a Markov process model for it. Observe that, at time $t$, the knowledge of all the current synaptic weights and the times of all the spikes fired in the network within the time interval $(t-\Theta,t]$ is all the information from the past and present that one needs to uniquely specify the probability distribution of the future evolution of the system.

Therefore, to obtain a Markovian description of the network, we denote by
$\hat{\nu}_k (t)$ the number of spikes fired by  external source $k$ in the time window $(t-\Theta,t]$, $k\le M$. If $\hat{\nu}_k (t)=0$, then we say that source $k$ was at the state
 \[
\widehat{X}_{k}(t)\equiv ( \widehat{X}_{k,1}(t),  \widehat{X}_{k,2}(t),  \widehat{X}_{k,3}(t),\ldots) = (0,0,0,\ldots)\in \R_+^{\N}
 \]
at time~$t$. If $\hat{\nu}_k (t)\ge 1$, we set
$
\widehat{X}_{k,1}(t):= \widehat{T}_k (t) -t+\Theta\in(0,\Theta],
$
which is the time till the  last spike fired by $k$ prior to the ``present" time $t$ disappears from the moving  window $(s-\Theta,s]$, $s\ge t,$ and then  we denote by
$
\widehat{X}_{k,2}(t):= \widehat{T}_k (\widehat{T}_k (t)-) -t+\Theta\in(0,\Theta],
$
the time till the second last spike fired by $k$ prior to time $t$ disappears from the moving  window $(s-\Theta,s]$,
and so on, so that in this case we always have
\[
\widehat{X}_{k}(t)=  ( \widehat{X}_{k,1}(t), \widehat{X}_{k,2}(t), \ldots ,\widehat{X}_{k,\hat{\nu}_k (t)}(t),0,0,\ldots),
 \quad k=1,\ldots, M,
\]
with $\widehat{X}_{k,1}(t) >\widehat{X}_{k,2}(t) >\cdots > \widehat{X}_{k,\hat{\nu}_k (t)}(t)>0$ a.s.\ (as having two spikes at exactly the same time is a zero probability event).

Likewise, the state of neuron $i$  is described by the vector
\[
 {X}_{i}(t)=  (  {X}_{i,1}(t),  {X}_{i,2}(t)\ldots , {X}_{i,  \nu_i (t)}(t),0,0,\ldots),
 \quad i=1,\ldots, N,
\]
with $ {X}_{i,1}(t) > {X}_{i,2}(t) >\cdots >  {X}_{i, {\nu}_i (t)}(t)>0$ a.s., ${\nu}_i (t)$ being the number of spikes fired by $i$ during  $(t-\Theta,t]$. Now the complete history of spikes within the time window $(t-\Theta,t]$ is described by the vector
\[
Z(t):=(\widehat{X}(t); X(t)) := (\widehat{X}_{1}(t),\widehat{X}_{2}(t),  \ldots ,\widehat{X}_{M}(t) ; {X}_{1}(t), X_2 (t), \ldots, X_N (t)).
\]

For Model~I, this vector will completely specify the state of the network. The state space for the process $Z$ will be taken to be
\[
S:= E^{M+N}, \quad\mbox{where}\quad  E:=\bigcup_{n\ge 0} E^{(n)},
 \quad
\]
is the union of simplices
\[
  E^{(n)}:= \bigl\{(x_1, x_2, \ldots )\in \R^{\N}_+ : \Theta \ge x_1 > x_2 >\cdots> x_n>0;\ x_{n+m}=0, \ m>0\bigr\} ,
\]
$n=0,1,2,\ldots$ Note that the $n$-dimensional simplex $E^{(n)}$ is a face of the $(n+1)$-dimensional one, $E^{(n+1)},$ $n\ge 0$. We will endow $S$ with the product  $\sigma$-algebra $\mathscr{S}:= \mathscr{C}^{\otimes(M+N)} $, where $\mathscr{C}$ is the trace of the cylindric $\sigma$-algebra on the space $\R^\N$
on~$E$.

For Model~II, we need to specify in addition the state of the synaptic connections. This can be done by using the matrices
\[
\widehat{W} (t) = (\widehat{W}_{ik}(t))_{i\le N, k\le N}, \quad  {W} (t) = (W_{ij}(t))_{i,j\le N},
\]
in which to non-existent connections there will correspond zero entries.
The new process $ Z^* := (\widehat{X}; X; \widehat{W}; W) $ will have the state space
\[
S^* := S\times \biggl(\prod_{i\le N, k\le M} \widehat{{G}}_{ik}\biggr)\times
 \biggl(\prod_{i,j\le N } {G}_{ij}\biggr),
\]
endowed with the natural product $\sigma$-algebra that we will denote by~$\mathscr{S}^*$.

\smallskip

{\bf Model I dynamics.} Suppose we are given an initial condition $Z(0)\in S.$
Following the earlier description of the dynamics of our network, in the case of Model~I the process $Z$ is a piece-wise deterministic (linear) Markov process, which evolves for $t>0$ as follows.

\smallskip

(i) Inside time intervals free of jumps, one has, for any $k\le M,$ $i\le N$, $n>1,$
\begin{equation}
  \frac{d\widehat{X}_{k,n}(t)}{dt} = - \mathbf{1} (\widehat{X}_{k,n}(t)>0),
  \quad
  \frac{dX_{i,n}(t)}{dt} = - \mathbf{1} (X_{i,n}(t)>0).
 \label{derX}
\end{equation}
This means that all the non-zero components of the process decay at the unit rate, and when the ``first visible in the window" spike of, say, neuron $i$ that occurred at time $T_{i,n}$ ``disappears" from the moving time window $(t-\Theta, t]$ at time $t'=T_{i,n}+\Theta$, the number of positive components of $X_i$ drops by one: $\nu_t(t')=\nu_t(t'-)-1.$

(ii) Given the state of the process is $\bz =(\hat\bx; \bx)\in S,$ where $\hat\bx =(\hat\bx_1, \ldots, \hat\bx_M)$ has components $\hat\bx_k = (\hat x_{k,1}, \ldots, \hat x_{k,m_k}, 0,0,\ldots )\in E^{(m_k)}$ with $m_k\ge 0,$ $k\le M,$ and likewise $ \bx =( \bx_1, \ldots, \bx_N)$ has   $ \bx_i = (  x_{i,1}, \ldots,   x_{i,n_i}, 0,0,\ldots )\in E^{(n_i)}$, $i\le N,$ the instantaneous firing rate for source $k$ is $\hat \rho_k$, and for neuron $i$ it is  given by
\begin{align}
R_i (\bz) & :=
\varsigma_i \biggl(   v_i +  \sum_{k\le M} \widehat{W}_{ik} \sum_{m\ge 1}     \hat{\epsilon}_{ik}(\Theta-\hat{x}_{k,m}) \notag
 \\
& \hphantom{ := \varsigma_i \biggl(   v_i}\
 +\sum_{j\le N} W_{ij} \sum_{n\ge 1}  \epsilon_{ij}(\Theta -x_{j,n})\biggr) r (\Theta -x_{i,1}).
  \label{Ri}
\end{align}

(iii) When source $k$ fires (say, at time $t'=\widehat{T}_{k,m}$), the only change in the state of $Z$ is in the component $\widehat{X}_k:$
\[
\hat \nu_{k} (t')  = \hat \nu_{k} (t'-)+1
\]
a.s.\ (as it is impossible to  simultaneously ``lose" a spike in the time window and acquire a new one), and the new values of the components are:
\begin{align*}
\widehat{X}_{k,1}(t' ) &=\Theta,
 \\
\widehat{X}_{k,2}(t')&=\widehat{X}_{k,1}(t'- ),
 \\
\widehat{X}_{k,3}(t' )&=\widehat{X}_{k,2}(t'- ),
 \\
&\cdots
 \\
\widehat{X}_{k, \nu_{k} (t')}(t' ) &=\widehat{X}_{k, \nu_{k}(t')-1}(t'- ).
\end{align*}
Likewise, a spike fired by neuron $i$ will mean similar changes in the component $X_i.$

\smallskip

It is quite straightforward to write down the generator of the process $Z$, following the above description. The vector field specifying the dynamics of the process between jumps is piece-wise linear, and it changes its direction when, for one of the components $\widehat{X}_k\in E$ or $X_i\in E$, the respective integral curve running inside $E^{(n)}$, $n>1$, hits the face $E^{(n-1)}$ of that simplex and then continues inside that lower dimensional simplex. The domain of the generator will consist of all bounded functions $S\mapsto \R$ that are path-continuous and differentiable for that vector field (cf.~\cite{Jacobsen06}).

\smallskip

{\bf Model II dynamics.}
The trajectories of $Z^*$ will also be piece-wise deterministic (linear), with its first two components following~\eqref{derX} and the last two remaining unchanged between successive jumps. Jumps occur at the times when either external sources or neurons fire spikes, and, given the current state of the process is $\bz^*=(\hat{\bx}; \bx; \hat\bw; \bw)$ with $ \hat\bw=(\hat w_{ik})$ and $\bw=(  w_{ij} )$, the instantaneous firing intensities are $\hat\rho_k$ for source $k$ and, instead of~\eqref{Ri},
\begin{align}
R_i^* (\bz^*) & :=
\varsigma_i \biggl(   v_i +  \sum_{k\le M} \hat{w}_{ik} \sum_{m\ge 1}     \hat{\epsilon}_{ik}(\Theta-\hat{x}_{k,m}) \notag
 \\
& \hphantom{ := \varsigma_i \biggl(   v_i}\
 +\sum_{j\le N} w_{ij} \sum_{n\ge 1}  \epsilon_{ij}(\Theta -x_{j,n})\biggr) r (\Theta -x_{i,1}).
  \label{Ri*}
\end{align}
for neuron $i$.

When a spike is fired, the change in the components $\widehat{X}$ and $X$ is exactly the same as for Model~I (see part~(iii) of the description of its dynamics above), whereas the synaptic weight $\widehat{W}_{ik}$ can change when the spike was fired either by source~$k$ or by neuron~$i$. As $\delta_{LW}<\Theta$, the state of $Z^*$ just prior to the spike completely specifies to what value the synaptic weight should change, according to the learning rules listed in the description of Model~II. Similarly for the weights ${W}_{ij}$ that  can change when the spike is fired by either of the neurons~$i$ and~$j$. Thus we see that, in the case of Model~II, $Z^*$ is a well-defined Markov continuous time process completely describing the dynamics of the system. Its generator will differ from the one for $Z$ by the presence of terms related to jumps in the synaptic weights' values.

\section{Ergodicity and the properties of stationary distributions}
\label{secErg}

In this section, we establish strong ergodicity of the Markov process $Z^*$ describing the dynamics of Model~II. As Model~I is a special case of the latter, this means that the process $Z$ is also ergodic. The latter fact is actually an immediate consequence of Theorem~5 in~\cite{BremMass96} on stability of multivariate point processes with bounded memory dynamics (see also Theorem~6 in~\cite{BremMass96} for stability of a nonlinear multivariate Hawkes process with PSP transfer kernels having unbounded supports, and~\cite{Mass98}). However, even in the case of Model~I, our Markov process framework allows us to come up with much shorter and simpler proof of stability.

\begin{thm}
 \label{Thm_erg}
Under the stated assumptions for Model~II, the process $Z^*$ is strongly ergodic: it has a unique stationary   distribution $\pi$ on $(S^*,\mathscr{S}^*)$ such that
\begin{equation}
 \label{tot_var}
 \sup_{\sbz^*\in S^*} \sup_{B\in \mathscr{S}^*}\bigl|\pr (Z^*(t)\in B\, |\, Z^*(0)=\bz^*)-\pi(B) \bigr|\to 0\ \textrm{ as } \ t\to\infty.
\end{equation}
Moreover, the convergence is exponentially fast.
\end{thm}

It would be most interesting to know the properties of the stationary distribution
$\pi$. One basic fact that we can easily establish is that  the distribution $\pi_S$ of
the first two components of $(\widehat X(\infty); X(\infty); \widehat{W}(\infty),
W(\infty) )\sim \pi$ on $(S, \mathscr{S})$ has a density  w.r.t.\ some naturally chosen
measure. Moreover, we can obtain upper bounds for the density.

More precisely, that natural measure on
$(S, \mathscr{S})$ is taken to be the product measure $\mu^{M+N},$ where
\[
\mu (B) = \sum_{n\ge 0} \mu_n (B_n)
 \quad\mbox{for}\quad B= \bigcup_{n\ge 0} (B_n \times \{\bo\}),\quad \bo=(0,0,\ldots)\in \R_+^{\N},
\]
$\mu_n $ being the $n$-dimensional Lebesgue measure and  $B_n$  being  Borel subsets of the respective $n$-dimensional simplices
\begin{equation}
 \label{E_0}
E^{(n)}_0:=  \bigl\{(x_1,  \ldots, x_n )\in \R^{n}_+ : \Theta \ge x_1 > x_2 >\cdots> x_n>0 \bigr\}
\end{equation}
that can be identified with $E^{(n)}=E^{(n)}_0\times \{\bo\}$. We use the convention that $\mu_0$ is just the unit mass at~0.

For $(\bm; \bn):=(m_1, \ldots, m_M;n_1, \ldots,   n_N )\in \Z_+^{M+N}$, set
\[
E^{(\sbm; \sbn )}:= \prod_{k\le M} E^{(m_k)} \times \prod_{i\le N}E^{(n_i)};
\]
similarly, $E^{(\sbm; \sbn )}_0$ is the product of the respective $E_0$-sets.

To simplify the formulation of the next theorem, we will slightly abuse notation by identifying the sets $E^{(\sbm; \sbn )} $ with $E^{(\sbm; \sbn )}_0$ and so considering the latter as the components of the state space $S$ (so that the components of $\mu$ are actually given on finite-dimensional spaces).

\begin{thm}
 \label{Thm_den}
Under the stated assumptions for Model~II, if all the functions $\hat\epsilon_{ik},$ $\epsilon_{ij},$
$\varsigma_i$ and~$r$ are continuously differentiable, then, for any $(\bm; \bn)\in \Z_+^{M+N}$, the restriction of $\pi_S$ to $E^{(\sbm; \sbn )}_0$ has a density $\psi_{\sbm, \sbn}$ w.r.t.\ $\mu$ admitting an upper bound
\[
\biggl(\prod_{k\le M} \hat\rho_k^{ m_k}\biggr)
 \biggl( \prod_{i\le  N} \overline\varsigma_i^{ n_i}\biggr)
\exp\left\{-\Theta\Sigma_{\hat\rho}\right\}\le \Lambda^{\Sigma_{{\mbox{\tiny\boldmath$m$}}} + \Sigma_{{\mbox{\tiny\boldmath$n$}}}}
\exp\left\{-\Theta\Sigma_{\hat\rho}\right\},
\]
where $\Sigma_{\hat \rho}:= \sum_k\hat \rho_k,$   $\Lambda:=\max\{\max_k\hat\rho_k,
\max_i\overline{\varsigma}_i \},$ $\Sigma_{\sbm} :=\sum_{k\le M}  m_k,$ and $ \Sigma_{\sbn}:=
\sum_{i\le N} n_i .$ The density function $\psi_{\sbm, \sbn}$ is continuously
differentiable in the interior of $E^{(\sbm, \sbn )}_0$ and has finite limits on its
boundary.
\end{thm}

\begin{rema}{\em As will be easily seen from the proof of Theorem~\ref{Thm_den}, if we assume that the function $r$ has form~\eqref{refr} (and so is not continuously differentiable), the assertion of the theorem will remain true with the only amendment that the density components $\psi_{\sbm, \sbn}$ will be continuously differentiable inside their supports in the spaces of the respective dimensionalities.
}
\end{rema}

\begin{proof}[Proof of Theorem~\ref{Thm_erg}]
It is obvious that $Z^*$ is aperiodic and stochastically continuous. So it suffices to
show that the Markov ``skeleton"  chain $Y=\{Y_n := Z^* ( \Theta n),$ $n=0,1,2,\ldots\},$ has a recurrent state whose first hitting time distribution tail decays exponentially fast uniformly in the chain's initial state~$Z^*(0)$  (see e.g.\ Theorem~18.1 in~\cite{Borov98}).

First we will use a standard argument to show that the tail of $\tau:= \inf  \{n>0: Z(
\Theta n) =(\mathbf{0};\mathbf{0})\}$ (i.e.\ the first value $n$ such that there were
no spikes in $(\Theta (n-1), \Theta n]$) admits such a bound. Indeed, setting for
convenience $\pr_{\sbz^*} (\cdot):= \pr (\cdot| Z^* (0)=\bz^*)$, we have, for any
$\bz^*\in S^*$ and $t>0$,
\begin{align}
\pr_{\sbz^*}  (\mbox{no spikes in } (0,t]) &=\exp\left\{
 -\int_0^t \biggl(\sum_k\hat \rho_k + \sum_i R_i^* (\cdots )\biggr)ds
\right\}\notag
\\
 & \ge \exp\{ -(\Sigma_{\hat \rho} + \Sigma_{\overline{\varsigma}})t\}
  =: e^{-\gamma t},
   \label{prBound}
\end{align}
where  $\Sigma_{\overline{\varsigma}}:=\sum_i \overline{\varsigma}_i,$
and $(\cdots)$ represents the argument of $R_i^*$ along the trajectory of $Z^*$ on $[0,t]$ that started at $\bz^*$ and experienced no jumps.

Now setting $A_n:= \{\mbox{no spikes in }(\Theta (n-1), \Theta n]\}$ we obtain, using
recursively the Markov property and bound~\eqref{prBound}, that,  for $ n\ge 0$,
\begin{align}
\pr_{\sbz^*} (\tau >n)
 & =
 \pr_{\sbz^*} \left(\bigcap_{m=1}^n A_m^c\right)
  = \exn  \pr_{\sbz^*}  \left(\bigcap_{m=1}^n A_m^c \bigg| Y_{n-1}\right)
   \notag
   \\
   & = \exn  \pr_{\sbz^*}  \left(\bigcap_{m=1}^{n-1} A_m^c \bigg| Y_{n-1}\right)
     \pr_{Y_{n-1}}  \left(  A_n^c  \right)
     \notag
     \\
     & \le (1-e^{-\gamma \Theta})  \exn  \pr_{\sbz^*}  \left(\bigcap_{m=1}^{n-1} A_m^c \bigg| Y_{n-1}\right) =
    (1-e^{-\gamma \Theta})  \pr_{\sbz^*} \left(\bigcap_{m=1}^{n-1} A_m^c\right)
   \notag
   \\
 &    \le \cdots \le  (1-e^{-\gamma \Theta})^n.
  \label{prBound_1}
\end{align}

Next we observe that $(\widehat{W} (\Theta n); {W} (\Theta n))$ is clearly an
indecomposable aperiodic finite Markov chain, and hence it is ergodic. Take any fixed
state $(\hat\bw'; \bw')$ of this chain; as it is well known, for any initial condition,
the first hitting time of that state has an exponentially fast decaying distribution
tail. Hence it is obvious that the state $(\mathbf{0};\mathbf{0};\hat\bw'; \bw')\in
S^*$ will be positive recurrent for the chain $Y $, and that the tail of the first hitting
time of that state by $Y$ will admit a geometrically fast vanishing upper bound uniform
in the initial condition of the state. The theorem is proved.
\end{proof}

\begin{proof}[Proof of Theorem~\ref{Thm_den}]
First we will establish existence of density for transition probabilities, and then infer the desired result from that fact.

Suppose our process started at point $Z^* (0)=\bv^*$ and, at time $\Theta,$ was at
a point $\bz^*=(\hat \bx;\bx;\hat \bw;\bw)\in S^*$ with  $\hat\bx_k\in E^{( m_k)},$
$k\le  M$, $\bx_i\in E^{( n_i)},$ $i\le N$. It is clear that the states  $\bv^*$ and $
\bz =(\hat \bx;\bx)\in S$ completely specify the {\em trajectory\/} of $Z^*(t)$ on the
time interval $[0,\Theta]$; denote this trajectory by $\bu (t),$ $t\in [0,\Theta]$ (so
that $\bu (0)= \bv^*$ and $\bu (\Theta)=\bz^*$). Then, observing that
$x_{i,1},\ldots,x_{i,n_i}$ are the firing times for neuron $i$ in the time interval
$[0,\Theta]$, we use the standard argument to show that
\begin{align}
\pr_{\sbv^*} &\bigl(Z(\Theta)   \in d\hat \bx_1 \times \cdots  \times d\hat \bx_M
  \times d  \bx_1 \times\cdots  \times d  \bx_N\bigr)
   \notag
   \\
    &=\left(\prod_{k\le M}\hat\rho_k^{ m_k}e^{- \hat\rho_k \Theta}\right) \mu_{  m_1 } (d\hat \bx_1)\cdots \mu_{  m_M } (d\hat \bx_M)
    \notag
     \\
     & \hspace{10mm}\times \left[
     \prod_{i\le N}  \left(\prod_{l_i\le n_i} R_i^* (\bu (x_{i,l_i}-)) \right)
     \exp \left\{-\int_0^\Theta R_i^* (\bu (t) )dt\right\}
     \right]
     \notag
     \\
     & \hspace{10mm}\times
     \mu_{   n_1 } (d  \bx_1)\cdots \mu_{   n_N } (d  \bx_N)
     \notag
     \\
     & =: p (\bv^* , \bz ) \mu_{  m_1 } (d\hat \bx_1)\cdots \mu_{  m_M } (d\hat \bx_M)
    \mu_{   n_1 } (d  \bx_1)\cdots \mu_{   n_N } (d  \bx_N)
    \notag
    \\
     & = p (\bv^* , \bz ) \mu(d\bz).
    \label{dens}
\end{align}
Clearly, the function $p (\bv^*, \bz )$ is continuously differentiable in $\bz$ in the
interior of $E^{(\sbm, \sbn )}$, has finite limits on its boundary, and admits an upper
bound of the form
\begin{equation}
p (\bv^*, \bz ) \le
 \biggl(\prod_{k\le M} \hat\rho_k^{ m_k}\biggr)
 \biggl( \prod_{i\le  N} \overline\varsigma_i^{ n_i}\biggr)
\exp\left\{-\Theta\Sigma_{\hat\rho}\right\}.
    \label{bound}
\end{equation}

Next, in view of~\eqref{dens}, for any $B\in\mathscr{S}$, we can use Fubini's theorem to write
\begin{align}
\pi_S (B) &= \int_{S^*}\pi (d \bv^*)\int_B p (\bv^*, \bz )\mu (d\bz)
 \notag
 \\
 &=  \int_B   \biggl[\int_{S^*}\pi (d\bv^* )p (\bv^*, \bz )\biggr]\mu (d\bz)=:
 \int_B \psi (\bz) \mu (d\bz).
\label{dens_stat}
\end{align}
This means that $\pi_S$ does have density $\psi$ w.r.t.\ $\mu$, and \eqref{bound} implies that $\psi$ admits the desired upper bound.

That $\psi$ is continuously differentiable in the relative  interiors of  the
components of its supporting space follows from representation~\eqref{dens}, the last
relation in~\eqref{dens_stat} and the assumption that $\varsigma_i$ and the kernel
functions $\epsilon_{ij}$ are all continuously differentiable, the $\epsilon$'s
vanishing outside $[0,\Theta].$ Theorem~\ref{Thm_den} is proved.
\end{proof}

\section{Approximation of $\pi$ by finite-dimensional distributions}
\label{appra}

In this section we will be dealing with the simpler Model~I. Even for that model,  the
state space is an infinite hierarchy of multidimensional simplices, so working with
non-trivial distributions on it is not easy. The natural question in such a situation
is whether one can find an appropriate approximation to the distribution in question,
together with an approximation error bound.

For our model, a tempting approach to finding such approximations is to consider
``truncated" processes $Z^{\langle n \rangle}$ in which none of the neurons is
``allowed" to fire more than the fixed number $n\ge 1$ times  within any given time
interval of length~$\Theta$. In fact, if the model assumes existence of absolute
refractory periods of positive length by stipulating, say, that~\eqref{refr} holds true,
then that condition will automatically be satisfied
(note, however, that one can still apply truncation with $n<\Theta/\delta_{AR}$ to reduce
dimensionality).

The only difference in the dynamics of the process $Z^{ \langle n \rangle }$ compared
to  those of $Z$ is that neurons' firing intensities will now be given by
\begin{equation}
R_i^{ \langle n \rangle } (\bz) := R_i  (\bz) \mathbf{1} (x_{i,n}=0), \quad i=1,\ldots, N,
 \label{Rin}
\end{equation}
(cf.~\eqref{Ri}). It is obvious that  $Z^{ \langle n \rangle }$ will also be an ergodic
Markov process. Denote its stationary distribution on $(S,\mathscr{S})$ by $\pi^{
\langle n \rangle }$, while for the stationary distribution of $Z$ (on the same
measurable space) we will re-use notation~$\pi$.

\begin{thm}
 \label{Thm_appr}
Under the stated assumptions for Model~I,
\begin{equation}
 \label{appr}
\sup_{B\in \mathscr{S} }\bigl|\pi( B) - \pi^{ \langle n \rangle }( B) \bigr|
 \le C n^{-(n+1)/2} e^{\alpha n},
 \end{equation}
where $ C=\frac{2N}{\sqrt{\pi}} \exp\{\Theta (\Sigma_{\hat \rho} + \Sigma_{\overline \varsigma})\}$ and $\alpha =(1 + \ln (\Theta \overline{\varsigma}))/2.$
\end{thm}

\begin{proof} We will use coupling. Assume that $\Pi$ is a Poisson random field of unit intensity on $\R_+\times \R$, given on some probability space,
and construct a process $\{(Z (t), Z^{ \langle n \rangle }(t))\}_{t\ge 0}$ with the
state space $S\times S$, whose components follow the original and ``truncated"
dynamics, respectively, start at a common state $Z (0)= Z^{ \langle n \rangle }(0)\in
S$, and are driven by the field $\Pi$ via the following simple mechanism.

Introduce intervals
\begin{align*}
\widehat I_k  & := \biggl(-\sum_{m=1}^k \hat\rho_m, -\sum_{m=1}^{k-1} \hat\rho_m\biggr],
\quad k=1,\ldots, M,\\
 I_i &  :=  ( \vartheta_{i-1}, \vartheta_i],
\quad i=1,\ldots, N, \quad \vartheta_i:= \sum_{j=1}^{i} \overline{\varsigma}_j,
\end{align*}
and stipulate that, in both $Z $ and $ Z^{ \langle n \rangle },$ external source $k$
fires at time $t$ if $\Pi (\{t\}\times \widehat I_k)>0$ (note that $\{\Pi ([0,t]\times
\widehat I_k)\}_{t\ge 0}$ are independent Poisson processes with constant intensities $\hat
\rho_k,$ $k=1,\ldots, M$).

Likewise, in the process $Z$ neuron $i$ fires at time $t$ if
\[
\Pi \bigl(\{t\}\times (  \vartheta_{i-1}, \vartheta_{i-1}+ R_i (Z(t-))]\bigr)>0,
\]
and that happens in the process $Z^{ \langle n \rangle }$ at time $t$ if
\[
\Pi \bigl(\{t\}\times ( \vartheta_{i-1} ,\vartheta_{i-1}+ R_i^{ \langle n \rangle }(Z^{
\langle n \rangle }(t-))]\bigr)>0.
\]
Clearly,  $(Z ,Z^{ \langle n \rangle })$ is a well-defined Markov process, and its
components follow the desired dynamics. Note also that the process will be ergodic,
like each of its components (this is obvious e.g.\ from Theorem~\ref{Thm_erg}).

Now denote by $(Z (\infty), Z^{ \langle n \rangle }(\infty))$ a random element of
$S\times S$ whose distribution coincides with the stationary distribution of $(Z ,Z^{
\langle n \rangle })$. Using the standard argument, it is easily seen that
\[
\bigl|\pi( B) - \pi^{ \langle n \rangle }( B) \bigr| \le \pr (Z (\infty)\neq Z^{
\langle n \rangle }(\infty))=:P_n,\quad
 B\in\mathscr{S}.
\]
To bound $P_n$, denote by $|T_K|,$ $K\in\N,$ the total length of the set
\[
T_K := \{ t\in [0, \Theta K]: \, Z (t) \neq Z^{ \langle n \rangle }(t)\}
\]
and observe that, from the ergodicity of  $(Z ,Z^{ \langle n \rangle })$, one has
\[
P_n =\lim_{K\to\infty} \frac{|T_K|}{\Theta K}.
\]

As we are interested in the stationary distribution, we can assume w.l.o.g.\ that the
common starting point of $Z$ and $Z^{ \langle n \rangle }$ has  no components $\bx_i$
in $E^{(m)},$ $m\ge n$. Then the trajectories  $Z(t)$ and $Z^{ \langle n \rangle }(t)$,
having originated at the same point,  will coincide with each other till the time $T'$
when one of the values $X_i (t),$ $i=1,\ldots, N,$ enters $E^{(n)}$. Then the
respective neuron $i$ will stay silent in $Z^{ \langle n \rangle }$ at least till the
time when the number of spikes produced by $i$ and ``visible" in the time window
$(t-\Theta, t]$ drops below~$n$, while in $Z$ the respective neuron will still be able
to fire. Thus, past that time point $T'$, the trajectories $Z(t)$ and $Z^{ \langle n
\rangle }(t)$ can diverge. They will have to meet again, though, and the latest that
will occur is at the end of the next ``silent interval" of length~$\Theta$, which, in
its turn, occurs no later than at the time
\[
\inf\bigl\{t>T':  \Pi \bigl((t-\Theta, t] \times (-\Sigma_{\hat \rho}, \Sigma_{\overline \varsigma}]=0\bigr)\bigr\}.
\]

To make use of the above argument to obtain an upper bound for $P_n$,  introduce
two random sequences, $\{\varkappa_m\}_{m\ge 0}$ and $\{\gamma_m\}_{m\ge 0}$, as
follows.  Letting for brevity $\theta_j:= \Theta j,$ set
\begin{align*}
V_{i,j-} &:= \Pi \bigl( (\theta_{j-1},\theta_{j-1} + \Theta/2]
 \times I_i  \bigr),
  \\
V_{i,j+}& := \Pi \bigl( (\theta_{j-1} + \Theta/2,\theta_{j} ]
 \times I_i  \bigr),
\end{align*}
and then put  $\varkappa_0 := \gamma_0 :=0$ and, for $m\ge 1,$
\begin{align*}
\varkappa_{m+1}& := \inf\Bigl\{
 j> \gamma_m: \, \max_{i\le N} \max\bigl\{V_{i,j-}, V_{i,j+}\bigr\}\ge  {n}/{2}\}
\Bigr\}, \\
 \gamma_{m+1}& := \inf\Bigl\{
 j> \varkappa_{m+1}: \,
  \Pi \bigl((\theta_{j-1},\theta_{j}] \times (-\Sigma_{\hat \rho}, \Sigma_{\overline{\varsigma}}]\bigr) =0
\Bigr\}.
\end{align*}
Clearly, both  $\{\varkappa_m\}$ and $\{\gamma_m\}$ are well-defined a.s.\ infinite increasing sequences of proper random variables.

Now if  $Z(\theta_{j-1})=Z^{ \langle n \rangle }(\theta_{j-1})$ but, for some $t\in
(\theta_{j-1},\theta_{j}]$, one has $Z(t)\neq Z^{ \langle n \rangle }(t)$, then, for
some $i\le N,$ at least one of the following two relations must hold:
\[
\max\bigl\{V_{i,(j-1)-}, V_{i,(j-1)+}\bigr\} \ge \frac{n}{2}, \quad
 \max\bigl\{V_{i,j-}, V_{i,j+}\bigr\} \ge \frac{n}{2}
\]
(if none of the two holds  then, in any time interval of length $\Theta$ within
$(\theta_{j-2}, \theta_j]$, neuron $i$ will have fewer than~$n$ spikes). Thus the
values $\varkappa_m$ ``mark" time intervals where $Z$ and $Z^{ \langle n \rangle }$ may
split, whereas $\gamma_m$ ``mark" those intervals following $\varkappa_m$ where $Z$ and
$Z^{ \langle n \rangle }$ must merge (provided that they have split indeed).

Set
\[
H_K := \inf\{m\ge 1: \, \varkappa_m >K\} -1.
\]
Clearly, for $t>0$ and an arbitrary fixed $\varepsilon >0$,
\begin{equation}
\pr (|T_K|> t) \le \pr (|T_K|> t, H_K \le \varepsilon K) +
 \pr (H_K > \varepsilon K) .
 \label{TH}
\end{equation}
First we will bound the last term. Observe that
\[
H_K \le \sum_{j=1}^K \chi_j, \quad \mbox{where}\quad
 \chi_j := \mathbf{1} \Bigl(\max_{i\le N} \max\{V_{i,j-}, V_{i,j+}\} \ge n/2\Bigr)
\]
are i.i.d.\ Bernoulli random variables with success probability
\[
 \pr (\chi_j =1) = \pr \Bigl(\max_{i\le N} \max\{V_{i,j-}, V_{i,j+}\} \ge n/2\Bigr)
 \le 2\sum_{i\le N} \pr (V_{i,j-}\ge n/2) .
\]
Since $V_{i,j-}$ has the Poisson distribution with parameter $\lambda_i:=\Theta\overline{\varsigma}_i/2$, one can use Taylor's formula for the exponential series (with remainder in Lagrange form) and then Stirling's formula to write
\begin{align*}
\pr (V_{i,j-}\ge n/2) & \le \frac{\lambda_i^{n/2}}{(n/2)!}
 \le \frac{1}{\sqrt{2\pi}} \biggl(\frac{n}{2}\biggr)^{-(n+1)/2}
 \exp\Bigl\{\frac{n}{2}(1 +\ln \lambda_i)\Bigr\}\\
 & \le \frac{1}{\sqrt{\pi}} \,  n^{-(n+1)/2}
  e^{\alpha n}.
\end{align*}
Therefore
\[
 \pr (\chi_j =1)\le \frac{2N}{\sqrt{\pi}} \,  n^{-(n+1)/2} e^{\alpha n} =: p_n.
\]
Assuming that $p_n<1$ (otherwise the bound in the theorem will become trivial),  and
that $\delta:= \varepsilon - p_n \equiv \varepsilon - \exn \chi_j >0$, we obtain
\begin{align}
\pr (H_K > \varepsilon K) & = \pr (H_K - K\exn \chi_1> \delta K)
 \notag
 \\
 & \le \pr \biggl(\sum_{j\le K} (\chi_j - \exn \chi_j)> \delta K\biggr)
  \le e^{-2 \delta^2 K}
 \label{HK}
\end{align}
by virtue of Theorem~10 from Chapter~5 of~\cite{Borov99}.

Now we will turn to the first term on the RHS of~\eqref{TH}. From the definitions of our random variables, it is obvious that $|T_K|\le \Theta\sum_{m\le H_K} (\gamma_m - \varkappa_m)$, and so \[
\pr (|T_K|> t, H_K \le \varepsilon K)\le
 \pr \biggl(\sum_{m\le \varepsilon K} (\gamma_m - \varkappa_m)> \frac{t}{\Theta}\biggr)
 =: Q.
\]
From the strong Markov property it follows that $\eta_m:= \gamma_m - \varkappa_m$ are i.i.d.\ geometric random variables, with $\pr(\eta_1 = k)=q (1-q)^{k-1},$ $k=1,2,\ldots,$ where
\[
q:= \pr \Bigl( \Pi \bigl((t-\Theta, t] \times (-\Sigma_{\hat \rho}, \Sigma_{\overline \varsigma}]\bigr)=0\Bigr) = \exp\{-\Theta(\Sigma_{\hat \rho} +\Sigma_{\overline \varsigma})\}.
\]
Clearly, $\exn \eta_1 = 1/q$ and $\varphi (a):=\exn e^{a\eta_1}<\infty $ for $a< -\ln (1- q),$
\begin{equation}
\varphi (a) = 1 +\frac{a}{q} + o(1), \qquad a\to 0.
 \label{phi}
\end{equation}
Therefore, assuming w.l.o.g.\ that $\varepsilon K$ is integer, we have by the exponential Chebyshev's inequality that
\[
Q \le \bigl(\varphi (a) \bigr)^{\varepsilon K}e^{-a t/\Theta}
 = \exp\biggl\{ - \varepsilon K \biggl(\frac{a t}{\Theta \varepsilon K} -\ln\varphi (a)\biggr) \biggr\}.
\]
One can see from~\eqref{phi} that, choosing $t=t_K:= \Theta
\varepsilon K(1+h)/q$ for an arbitrary  fixed $h>0$, we will have, for small enough $a$, the bound
\[
Q\le e^{-\varepsilon cK} \ \mbox{ for some } \ c=c(a,h).
\]

From here, \eqref{TH} and \eqref{HK}  we obtain the bound
\[
\pr (|T_K|> t_K) \le  e^{-\varepsilon cK} + e^{-2 \delta^2 K}.
\]
Clearly, $\sum_K \pr (|T_K|> t_K)<\infty$, and so, by Borel-Cantelli lemma, with
probability one  we have $|T_K |\le t_K$ for all large enough $K$. Therefore,
\[
P_n \le \limsup_{K\to\infty} \frac{t_K}{K\Theta} =\frac{\varepsilon (1+h)}{q}.
\]
As this holds for any $\varepsilon >p_n$ and $h>0,$ we conclude that $P_n\le p_n/q,$ which completes the proof of the theorem.\end{proof}

\section{Computing the stationary distribution}
\label{secComp}

It is not difficult to derive differential equations (and boundary conditions for them)
that the components $\psi_{\sbm, \sbn}$ of the stationary density of $Z$ will satisfy in the
case of Model~I. They may be derived from the general relation
\begin{equation}
\label{gen_eq}
 \exn Af (Z(\infty))=0,
\end{equation}
where $A$ is the infinitesimal generator of the process, $f$ a function from a suitable subset of the domain of~$A$, and, as before, $Z(\infty)\sim \pi$. It may be easier, however, to obtain them via
a direct argument, making use of our Theorem~\ref{Thm_den} (of which the conditions will be assumed satisfied in this section unless we explicitly state otherwise).

To show how to do that, we
will first consider the simple case of a network with one external source and one neuron (with feedback).
Suppose that the neuron has an absolute refractory period (so that the state space is actually finite-dimensional). For simplicity, we assume throughout this section that $\Theta=1$ (which clearly does not restrict generality).

In this case, the state space of the process  is just $E\times E,$ so that each state
$(\hat \bx;\bx)= (\hat x_1, \hat x_2, \ldots, \hat x_m, 0,0,\ldots;   x_1,   x_2,
\ldots,   x_n, 0,0,\ldots )$ (note that here we suppress  the unnecessary   first
subscript indicating the number of the external  source or neuron; likewise, $\hat
\rho$ will denote here $\hat\rho_1$ etc.) belongs to one of the components
$E^{(m)}\times E^{(n) },$ $m,n\ge 0.$ The respective density components we will denote
by $\psi_{ m,n}.$

The first of them, $\psi_{ 0,0},$ is just the stationary probability of the silent state, for which we have, for $\delta \searrow 0,$
\begin{align*}
\psi_{ 0,0}
 & =
 \pr (Z(\delta) = (0,0))
  \\
 & = \pr (Z(\delta) = (0,0) \,|\, Z(0) = (0,0)) \pr (Z(0) = (0,0))
  \\
 & \quad + \int_{0}^\delta \pr (Z(\delta) = (0,0) \,|\, Z(0) = (y,0))\psi_{1,0} (y) dy
  \\
   & \quad + \int_{0}^\delta  \pr (Z(\delta) = (0,0) \,|\, Z(0) = (0,y))\psi_{0,1} (y) dy +
   O(\delta^2)
   \\
   & = e^{- (\hat \rho + \varsigma (v))\delta} \psi_{ 0,0}
    + \int_{0}^\delta  (1 + o(1))\psi_{1,0} (y) dy  + \int_{0}^\delta  (1 + o(1))\psi_{0,1} (y)
    dy+ o(\delta),
\end{align*}
where the term $O(\delta^2)$ corresponds to the possibility that $Z(0)\in E^{(m)}\times
E^{(n) }$ with $m+n >1.$ From the above representation we obtain that
\begin{equation}
(\hat \rho + \varsigma (v)) \psi_{ 0,0} = \psi_{1,0} (0) + \psi_{0,1} (0),
 \label{psi_0}
\end{equation}
where, using Theorem~\ref{Thm_den}, we put $\psi_{1,0} (0):= \psi_{1,0} (0+)$,
$\psi_{0,1} (0):= \psi_{0,1} (0+)$.

In the case where $mn>0$, we fix a point  $\bz = (\hat \bx; \bx)$ in the interior of $ E^{(m)}_0\times E^{(n) }_0$ (see~\eqref{E_0}) and set $I_{\sbz} (\delta) := I_{\hat \sbx} (\delta) \times I_{  \sbx} (\delta), $ where
\begin{align*}
I_{\hat \sbx} (\delta) &:= (\hat x_1, \hat x_1 + \delta)
 \times \cdots \times(\hat x_m, \hat x_m + \delta),
 \\
 I_{\sbx} (\delta) &:=  (  x_1,   x_1 + \delta) \times \cdots \times(  x_n,   x_n + \delta)
\end{align*}
and $\delta >0$ is small enough so that  $I_{\sbz} (\delta) \subset  E^{(m)}_0\times E^{(n) }_0$. Using notation
$\bz +\theta$ for shifting all the components of the vector $\bz$ by the same amount  $\theta\in\R$ and, as we did it before, slightly abusing notation by identifying  $E^{(m)} \times E^{(n) } $ with $E^{(m)}_0\times E^{(n) }_0,$ we have
\begin{align*}
\pr (Z(\delta) \in I_{\sbz} (\delta))
  & =  \pr \bigl(Z(\delta) \in I_{\sbz} (\delta)\, | \, Z(0) \in I_{\sbz}(\delta) + \delta\bigr) \,
   \pr (Z(0) \in I_{\sbz}(\delta)+ \delta )
    \\
    &   \quad + \pr \bigl(Z(\delta) \in I_{\sbz} (\delta),\,
     Z(0) \in [(I_{\hat \sbx} (\delta)+\delta)\times (0, \delta)]\times  (I_{\sbx} (\delta)+\delta)\bigr)
     \\
    &   \quad + \pr \bigl(Z(\delta) \in I_{\sbz} (\delta),\,
     Z(0) \in (I_{\hat\sbx} (\delta)+\delta)
     \times [(I_{  \sbx} (\delta)+\delta) \times (0, \delta)]\bigr)
     \\
    &   \quad + O(\delta^{m+n+2}), \qquad \delta \searrow 0,
\end{align*}
where the last term corresponds to the possibility of $Z(0)$ being in a space of dimensionality higher than $n+m+1$. Expressing the probabilities above as integrals of the respective density components and using Theorem~\ref{Thm_den}, we obtain the relation
\begin{align*}
\int_{I_{\tbz}   (\delta)} & \psi_{m,n} (\by)    (\mu_m  \otimes\mu_n ) (d\by)
 \\
  &  = \int_{I_{\tbz} (\delta)} [1 - \delta(\hat \rho + R (\by+\delta ))   ]\psi_{m,n} (\by + \delta)  (\mu_m \otimes\mu_n ) (d\by) \\
   & \quad + (1 + o(1))  \delta^{m+n+1} [ \psi_{m+1,n}((\hat \by,0; \by)) +  \psi_{m,n+1}((\hat \by; \by ,0))]  + O(\delta^{m+n+2}).
\end{align*}
Subtracting from both sides the integral of $\psi_{m,n} (\by + \delta) $ over $I_{\sbz}$ one can then easily verify that the relation implies that the following differential  equation must be satisfied: for $\theta\in (0, 1 - \max\{\hat x_1, x_1\}),$
\begin{align}
\frac{\partial}{\partial \theta}\, \psi_{m,n} (\bz + \theta)
 & = (\hat \rho + R (\bz + \theta))\psi_{m,n} (\bz+ \theta)
 \notag\\
 & -
    \psi_{m+1,n}((\hat \bx+ \theta,0; \bx+ \theta )) -  \psi_{m,n+1}((\hat \bx+ \theta; \bx + \theta ,0)) .
    \label{DE_psi}
\end{align}
Of course, the equation will hold along the whole interval formed by the intersection of $E^{(m)}_0\times E^{(n) }_0$ with the straight line passing through the point $\bz$ and having the directional vector ${\bf e}_{m+n} :=(1,\ldots, 1)\in \R^{m+n}$, the boundary condition at its right point being specified by the rates of transition to $E^{(m)}_0\times E^{(n) }_0$ from the state space components of lower dimensionalities. For example, if $\hat x_1 < x_1$, then the right end point for the interval of validity of \eqref{DE_psi} corresponds to the point where the ray $ \bx + \theta {\bf e}_n,$ $\theta >0$, hits the ``right" face of $E^{(n)}_0$ (the point $ \hat\bx + \theta {\bf e}_m$ still being in the interior of~$E^{(m) }_0$). At that location, the system can only enter the component $E^{(m)}_0\times E^{(n) }_0$ by a jump from
\[
 (\hat \bx + 1 -  x_1; \bx^* + 1 -  x_1)
  \in  E^{(m )}_0\times E^{(n-1)}_0, \quad
  \mbox{where}\quad   \bx^*: = (  x_2,   x_3, \ldots,   x_n),
\]
caused by a new spike fired by the neuron. Using a probabilistic argument similar to the one above, it is easy to see that   the following must hold:
\begin{equation}
\psi_{m,n} (\bz + 1  -   x_1)  = R ( \hat \bx  + 1 -   x_1; \bx^* + 1 -  x_1)  \psi_{m,n-1} (\hat \bx + 1 -   x_1; \bx^* + 1 -   x_1).
\label{BoCon}
\end{equation}
A similar equation will hold in the case where $\hat x_1 > x_1$, but then the coefficient of $\psi_{m-1,n}$ on the right hand side of the respective relation will simply be $\hat \rho$. The case where only one of $m,n$ is zero is treated similarly.

Solving equations of the form \eqref{DE_psi} with boundary conditions \eqref{BoCon}, complemented by \eqref{psi_0} and the condition that $\sum_{m,n}\int \psi_{m,n}d (\mu_m \otimes \mu_n)=1$, is hardly possible except for the simplest cases. One such case
is considered in the following example.

\begin{exam}{\em
Consider the case of a single neuron with feedback and no external sources. Moreover, assume that the firing rate function has the property
\begin{equation}
R(\bx)=0\quad\mbox{ for all }\quad \bx\in E^{(n)},\ n\ge 2,
 \label{R=0}
\end{equation}
so that there cannot be more than two spikes in any given time interval of length $\Theta=1$ (say, due to the length of the absolutely refractory period exceeding 1/2). Thus the state space of the system is just $E^{(0)}\times E^{(1)}\times E^{(2)}$ (which we again can and will identify with $E^{(0)}_0\times E^{(1)}_0 \times E^{(2)}_0$), the density components being $\psi_n,$ $n=0,1,2$ (for $n>2$, all $\psi_n\equiv 0$).

Using an obvious notational convention, we see that an analog of \eqref{psi_0} in this  case has the form
\begin{equation}
R(0) \psi_0 = \psi_1 (0),
\label{eq_psi_0}
\end{equation}
while an analog of \eqref{DE_psi} is, in the case $n=1$,
\begin{equation}
\frac{d \psi_1(\theta)}{d \theta}
   = R(\theta) \psi_1 (\theta) - \psi_2 (\theta,0), \quad \theta \in (0,1),
\label{eq_psi_1}
\end{equation}
with the boundary condition (an analog of \eqref{BoCon})
\begin{equation}
\psi_1 (1) = R(0)\psi_0.
\label{BoCon_psi_1}
\end{equation}
When $n=2$, an analog of \eqref{DE_psi}     has the following  form: for any $y\in (0,1),$
\[
\frac{\partial \psi_2 (y+\theta, \theta)}{\partial \theta}
   =  R(y+\theta, \theta) \psi_2 (y+\theta, \theta) - \psi_3 (y+\theta, \theta ,0)\equiv 0,   \quad \theta \in (0,1-y),
\]
the right-hand side of the equation being zero due to \eqref{R=0}, with the boundary condition (again an analog of \eqref{BoCon})
\begin{equation}
\psi_2 (1,1-y) = R(1-y)\psi_1 (1-y).
\label{psi_2}
\end{equation}
The last two relations immediately imply that, for any $y\in (0,1),$
\[
\psi_2 (y+\theta, \theta) =   R(1-y)\psi_1 (1-y), \quad \theta \in (0,1-y).
\]
Therefore  $\psi_2 ( \theta, 0) =   R(1-\theta)\psi_1 (1-\theta)$, so that \eqref{eq_psi_1} becomes
\[
\frac{d \psi_1 (\theta)}{d \theta}
   = R(\theta) \psi_1 (\theta) -  R(1-\theta)\psi_1 (1-\theta) , \quad \theta \in (0,1).
  \]
This means that the function $\psi_1$ is symmetric about the point $\theta=1/2$, so that  $\psi_1 (\theta)= \psi_1 (1- \theta),$   $\theta \in (0,1)$ (hence conditions \eqref{eq_psi_0} and \eqref{BoCon_psi_1} are consistent) and the last differential equation can be re-written as
\begin{equation}
\frac{d \psi_1 (\theta)}{d \theta}
   = (R(\theta)  -  R(1-\theta))\psi_1 ( \theta) , \quad \theta \in (0,1).
\label{eq_psi_11}
\end{equation}
Setting $\varphi (\theta) : =\exp\left\{\int_0^\theta (R(y)  -  R(1-y))\,dy \right\},$ we derive from \eqref{eq_psi_0} and \eqref{eq_psi_11}  that
\[
\psi_1 (\theta) = R(0)\psi_0 \varphi (\theta), \quad \theta \in (0,1).
\]
Together with \eqref{psi_2} this completely specifies the density function $\psi$ (computing  $\psi_0$ is trivial).

 }
\end{exam}

\begin{exam}{\em
One can also obtain a closed form solution in the case of a single neuron with feedback and no absolutely refractory period, but under the special assumption that the neuron's  PSP kernel is exponential:  $\epsilon (t) =e^{-\alpha t}\mathbf{1} (t \ge 0) $   for some $\alpha >0$. Of course, this   violates the basic assumption  that $\epsilon$ has a finite support and is smooth, but such a simplified mathematical model could still provide a useful approximation to more realistic ones.

It is not hard to see that in this case the dynamics of the system can be described by a univariate Markov process $Y(t):= \sum_n \epsilon(t-T_n) >0.$  Assuming without loss of generality that $\alpha=1$, one can see that the process $Y$ is driven by the Ornstein-Uhlenbeck type equation
\begin{equation*}\label{processY}
dY(t)=-Y(t)dt + dZ(t), \qquad t > 0,
\end{equation*}
$Z(t)$ being  a pure jump process with unit jumps and instantaneous jump rate $\gamma (Y(t))$, where $\gamma (y):= \varsigma (v+ Wy)$ and $W$ is the weight of the ``self-connection'' of our neuron. The infinitesimal generator $A$ of the process $Y$ is clearly
\begin{equation}\label{gene}
A f(x) = -xf'(x)+\gamma(x)(f(x+1)-f(x)), \qquad x>0;
\end{equation}
its domain's description can be found e.g.\ in Proposition~4 in~\cite{HarrR76}. It is not hard to see that $Y$ is ergodic (see e.g.~\cite{BorovLast08}) and so has a unique stationary distribution that we will again denote by~$\pi$.  Substituting \eqref{gene} into \eqref{gen_eq} (with $Z(\infty)$ replaced by $Y(\infty)$) yields
\begin{equation}
 \label{func_eqb}
\int_0^\infty yf'(y)\pi(dy) = \int_0^\infty \gamma(x)\left(\int_x^{x+1}f'(y)dy\right)\pi(dx) .
\end{equation}
Routine calculation now leads to
\[
\int_0^\infty yf'(y)\pi(dy) = \int_0^\infty f'(y)b(y)dy,\quad
b(y):=\int_{( y-1)^+}^y \gamma(x)\pi(dx),
\]
where we used notation $x^+:= \max\{ x, 0\}$ for the positive part of~$x$. As this equation holds for a large enough class of functions $f$ (see e.g.~\cite{HarrR76}) and $b$ is continuous and locally bounded, we conclude that $\pi$ has a locally bounded and continuous density $\psi$ on $(0,\infty)$, satisfying
\begin{equation}\label{stat_distrib}
y\psi(y)= \int_{( y-1)^+ }^y \gamma(x)\psi(x) dx, \quad  y>0.
\end{equation}
This equation can be solved recursively, on intervals  $J_n:=(n,n+1)$,  $n\geqslant 0$. Straightforward calculations show that the stationary density is given by
\[
\psi (y) := \varphi_n (y), \quad y\in J_n, \quad n=0,1,2,\ldots,
\]
where the $\varphi_n$ are found recursively as
\begin{align*}
\varphi_0(y)&= \psi(1)\exp\left(\int_1^y \dfrac{\gamma(x)-1}{x}\, dx\right), \qquad y\in J_0,\\
\varphi_n(y)&=\Phi_n[\varphi_{n-1}](y), \qquad y\in J_n,\quad n\ge  1,
\end{align*}
where, for $f$ defined on $J_{n-1}$ with a finite $f(n-)$, we denote by $\Phi_n[f](y),$ $y\in J_n,$  the solution $\phi$ of the problem
\[
\phi'(y)=\dfrac{\gamma(y)-1}{y}\phi(y)-\dfrac{\gamma(y-1)}{y}f(y),
 \quad y\in J_n, \quad \phi(n)=f(n-).
\]
The only unknown constant $\psi(1)$ is just the normalizing factor that is  to be determined from $\int_0^\infty \psi(y)dy=1$. At the integer points the density $\psi $ can be defined by continuity.

 }
\end{exam}

In the general case, one can only hope to compute approximations to the stationary distribution of the network. One way to do that is to discretise the state space and approximate the differential equations for the density components $\psi_{\sbm,\sbn}$ discussed at the beginning of this section with respective difference equations, and then to solve the latter. However, although the existence of solution to the original system of differential equation follows from Theorem~\ref{Thm_den}, establishing its uniqueness and also the convergence of the solutions to the systems of approximating difference equations presents a challenge. We will follow an alternative approach by first approximating  the original stochastic process with a sequence of finite ergodic Markov chains in discrete time, and then proving convergence of their stationary distributions to the desired stationary distribution of~$Z$. In the rest of the section, we will be dealing with our Model~I, but one can easily see that analogous results hold for Model~II as well. The only reason why we restrict ourselves to Model~I here is that the formulation of results for the more general model is much more cumbersome.

For $q\in \N$, set $h=h(q):=1/q$ and denote by $\kk{E}$ the collection of all vectors  $\bxi=(\xi_1,\xi_2,\ldots,\xi_q)\in (h\Z)^q,$ such that $1\ge \xi_1 > \xi_2 >\cdots > \xi_n >0$ and $\xi_{n+1}= \xi_{n+1}=\cdots = \xi_{q}=0$ for some $ n\in \{0,1,\ldots,q\}$. In particular, the null vector and $(1, 1-h, 1-2h,\ldots, h)$ both belong to $\kk{E}$. Let
\[
\kk{S}:= \kk{E}^{M+N}
\]
and denote by $\kk{F}$ the ``natural embedding" $ \kk{S} \mapsto S$ under which the components of the vector $\bzeta := (\hat\bxi_1, \ldots, \hat\bxi_M;  \bxi_1, \ldots,  \bxi_N)\in \kk{S}$   are concatenated with infinite strings of zeros  so that, say, $ \bxi_i=(\xi_{i,1}, \ldots, \xi_{i,q}),$ with the last positive component being $\xi_{i,n_i},$ $n_i\le q,$ becomes $  (\xi_{i,1}, \ldots, \xi_{i,q}, 0, 0, \ldots)\in E^{(n_i)}_0$, and by $\kk{F}^{-1}(B)$ the preimage of $B\in\mathscr{S}$ under the mapping~$\kk{F}$.
Finally, for $\bxi=(\xi_1,\xi_2,\ldots,\xi_q)\in\kk{E}$, let
\[
\kk{U}\bxi:= (1, (\xi_1 - h)^+, (\xi_2 - h)^+,\ldots, (\xi_{q-1} - h)^+)\in \kk{E}.
\]

Now consider a Markov chain
\[
\kk{Z} (s)
 %
 %
 = (\kk{\widehat X}_1(s),\ldots , \kk{\widehat X}_M(s);
   \kk{X}_1(s),\ldots , \kk{X}_N(s)),\quad s=0,1,2,\ldots,
 \]
in the (finite) state space $\kk S$
with  one-step transition probabilities specified as follows. Given the value   $\kk{Z} (s) = \bzeta  =(\hat\bxi_1, \ldots, \hat\bxi_M; \bxi_1, \ldots, \bxi_N)\in \kk{S},$   one has the following transitions for the components of the vector $\kk{Z}$:
\begin{equation}
 \label{tra_1}
\kk{\widehat X}_k(s+1) = \left\{
 \begin{array}{ll}
  (\hat\bxi_k -h)^+   & \mbox{with probability }\ 1 - h\hat \rho_k,\\
  \kk{U} \hat\bxi_k  & \mbox{with probability }\  h\hat \rho_k,
 \end{array}
\right.
\end{equation}
\begin{equation}
 \label{tra_2}
\kk{X}_i (s+1) = \left\{
 \begin{array}{ll}
  ( \bxi_i-h)^+ & \mbox{with probability }\ 1 - hR_i (  \bzeta),\\
  \kk{U}  \bxi_i  & \mbox{with probability }\  hR_i ( \bzeta),
 \end{array}
\right.
\end{equation}
where the operations of subtracting a scalar and taking positive parts are understood in the component-wise sense, and all the transitions occur
independently of each other for $k\le M,$ $i\le N.$  The transitions presented as the second options on the right-had sides of the above relations correspond to spike firing by the respective sources and/or neurons in the original model, and we will keep referring to these events as spikes in the case of the discrete model as well.

The next theorem provides a way for numerical calculation of the stationary distribution $\pi$ of our original process $Z.$ Endow $S$ with the topology of component-wise convergence and introduce the following notation. For $\bxi=(\xi_1,\xi_2,\ldots,\xi_q)\in\kk{E}$, denote by
\begin{align*}
V_0\bxi :=
 \left\{
 \begin{array}{ll}
 (\xi_1+h, \xi_2+h, \ldots,  \xi_n+h, 0, 0 ,\ldots, 0)\in\kk{E} & \mbox{if $\xi_1<1$,}
 \\
  (\xi_2+h, \xi_3+h, \ldots, \xi_n+h, 0 ,0, \ldots, 0)\in\kk{E} & \mbox{if $\xi_1=1$,}
 \end{array}
 \right.
\end{align*}
and
\begin{align*}
V_1\bxi :=
 \left\{
 \begin{array}{ll}
 (\xi_1+h, \xi_2+h, \ldots,  \xi_n+h, h, 0 ,\ldots, 0)\in\kk{E} & \mbox{if $\xi_1<1$,}
 \\
  (\xi_2+h, \xi_3+h, \ldots, \xi_n+h, h ,0, \ldots, 0)\in\kk{E} & \mbox{if $\xi_1=1$,}
 \end{array}
 \right.
\end{align*}
possible immediate ``precursors" for the state $\bxi$ of a given component of the Markov chain~$\kk{Z}$, i.e.\ the results of ``inverting" transitions in~\eqref{tra_1} and~\eqref{tra_2}.  It is not hard to see that the states
\[
V_{\hat\sbal, \sbal}  (\bzeta)
 := (V_{\hat\alpha_1} (\hat\bxi_1),\ldots, V_{\hat\alpha_M} (\hat\bxi_M);
 V_{ \alpha_1} ( \bxi_1),\ldots, V_{ \alpha_N} ( \bxi_N))
  \in \kk{S},
\]
where $\hat\balpha =(\hat\alpha_1, \ldots,\hat\alpha_M  ) \in \{0,1\}^M$ and  $\balpha=(\alpha_1, \ldots, \alpha_N  )\in \{0,1\}^N,$ exhaust  all possible precursors of the state~$\bzeta := (\hat\bxi_1, \ldots, \hat\bxi_M;  \bxi_1, \ldots,  \bxi_N)\in \kk{S}$ of our Markov chain, and that
\begin{align*}
p( \bzeta | V_{\hat\sbal, \sbal}  (\bzeta) )
   :&=  \Biggl[\prod_{k\le M} (h\hat\rho_k)^{\mathbf{1} (\hat\xi_k=1)}
  (1-h\hat\rho_k)^{  \mathbf{1} (\hat\xi_k<1)}\Biggr]
   \\
  & \times \Biggl[\prod_{i\le N} (hR_i(V_{\hat\sbal, \sbal}  (\bzeta)))^{\mathbf{1} ( \xi_i=1)}
  (1-hR_i(V_{\hat\sbal, \sbal}  (\bzeta)))^{  \mathbf{1} (  \xi_i<1)}\Biggr]
\end{align*}
are transition probabilities from those states to~$\bzeta$.

\begin{thm}
 \label{Thm_disc}
For any $q\in\N,$ the Markov chain $\{\kk{Z}(s) \}_{s\ge 0}$ is ergodic with stationary distribution $\kk \pi=\{\kk \pi (\bzeta),\  \bzeta\in\kk{S}\}$ satisfying the following system of linear algebraic equations:
\[
\kk \pi (\bzeta) = \sum_{(\hat\sbal; \sbal)\in \{0,1\}^{M+N}}
 \kk \pi (V_{\hat\sbal, \sbal}  (\bzeta)) p( \bzeta | V_{\hat\sbal, \sbal}  (\bzeta) ), \quad \bzeta\in \kk{S};
  \quad \sum_{\sbzeta\in\kk{S}} \kk{\pi} (\bzeta)=1.
\]
Moreover, as $q\to\infty$, the distributions $\kk{\pi}\circ\kk{F}^{-1}$ converge  weakly to the stationary distribution $\pi$ of~$Z$ .
\end{thm}

\begin{proof} That the chain $\kk{Z}$ is ergodic is obvious since it is finite, irreducible and aperiodic. The system of equations that  $\kk{\pi}$ is claimed to satisfy is just an explicit form of the usual matrix equation $\kk{\pi}=\kk{\pi}\, \kk{P}$ for stationary probabilities, $\kk{P}$ being the transition probabilities matrix of our chain. So we only need to prove the last claim of the theorem.

Recall that we used $Z^{\langle n \rangle}$ to denote a ``truncated version" of the process $Z$, of which the components $X_i$ cannot take  values  in spaces of dimensionality higher than~$n$ (see Section~\ref{appra}). Here we will use the same notation for a similarly ``truncated" versions where the components $\widehat X_k$ are likewise constrained. It is easy to see that the assertions of Theorem~\ref{Thm_appr} remains true in this case as well (with a different value for~$C$).

Denote by $\kk{Z}^{\langle n \rangle}$ a similarly truncated version of the chain $\kk{Z}$ and observe that a complete analog of Theorem~\ref{Thm_appr},  with the same bound as in~\eqref{appr} (of which the right-hand side does not depend on $q$), will hold true for that process as well.

Now fix an arbitrary $\varepsilon >0$ and choose $n$ so large that the right-hand side of~\eqref{appr} is less than~$\varepsilon$. That means that the stationary distributions $\pi$ and $\pi^{\langle n \rangle}$ of the processes $Z$ and $Z^{\langle n \rangle}$, respectively, will be $\varepsilon$-close in total variation, and the same will apply to the stationary distributions $\kk{\pi}$ and $\kk{\pi}^{\langle n \rangle}$ of the processes $\kk{Z}$ and $\kk{Z}^{\langle n \rangle},$ too, so that
\begin{equation}
  \sup_{q>0} \biggl[ \sup_{B\in \mathscr{S}  }
    \bigl| \pi ( B )-\pi^{\langle n\rangle}(B) \bigr|
    + \sup_{B\subset\kk{S} }| \kk{\pi}( B ) -\kk{\pi}^{\langle n\rangle}(B) \bigr|\biggr] <2\varepsilon.
 \label{pipi}
\end{equation}

This observation implies that it suffices to prove the claim of Theorem~\ref{Thm_disc} for the truncated processes $Z^{\langle n \rangle}$ and $\kk{Z}^{\langle n \rangle}$  that take values in the finite-dimensional space~$\R^K$, $K:= n(M+N),$. To simplify notation,   we will suppress the superscript ${\langle n \rangle}$ in the next two paragraphs, so that $Z$ will mean there $Z^{\langle n \rangle}$ etc.

To prove convergence of the stationary distributions, first assume that $Z  (0) =\kk{Z} (0)={\bf 0}\in\R^K$ and then observe that, as $q\to\infty,$  the distributions of the processes $\{\kk{Z} (\lfloor qt \rfloor)\}_{t\ge 0}$  weakly converge to that of $\{Z(t)\}_{t\ge 0}$ in the Skorokhod space  $D_{\R^K} [0,\infty)$ (see e.g.\ Section~5 in Chapter~2 in~\cite{EtKu}; by $\lfloor x \rfloor$ we denote the integral part of~$x$).  This can be seen, for instance, from Theorem~2.6 in Chapter~4 in~\cite{EtKu} (in fact, the purpose of the ``truncation" that we did above as the first step in the proof  was to make the state space of the processes locally compact, which is one of the conditions of the theorem). Indeed, extend the domain of the transition operator $\kk{T}$ of the chain $\kk{Z}$ to all bounded measurable functions $f$ defined on the state space of $Z$ ($=Z^{\langle n \rangle} \in \R^K$) by setting
\[
\kk{T} f (\bz):= \exn \bigl[f(\kk{Z} (s+1))|\, \kk{Z} (s) =h\lfloor q \bz\rfloor \bigr],
\]
where $\lfloor   \by\rfloor$ denotes the vector whose components are equal to the integral parts of the respective components of~$\by$, and let
\[
T(t) f(\bz) :=  \exn \bigl[f( Z  (u+t))| Z (u) =  \bz  \bigr] , \quad u,t >0,
\]
be the transition semigroup of~$Z$. Then the conditions of the above-mentioned theorem from~\cite{EtKu} will be met provided that we show that, for any continuous function $f$ on $\R^K$ and any $t>0$, one has
\begin{equation}
\lim_{q\to\infty}\sup_{\sbz} |\kk{T}^{\lfloor q t\rfloor} f (\bz) - T(t) f(\bz) |=0.
 \label{conva}
\end{equation}
Because of the semigroup property, it suffices to prove that convergence  for $t\in [0,1]$ only, which is not hard to do using representations of the form~\eqref{dens}.

Indeed, assume that $t=1$ (recall that we assumed that $\Theta=1$ here); the argument in the case $t<1$ will be similar, but we will need to integrate over subspaces then, which makes everything even more cumbersome. Partition the component $E^{(\sbm; \sbn )}_0$ of the domain of integration of $T(t) f(\bz)$ into cubes of edge length $h$ with vertices on the grid $(h\Z)^{\Sigma_{\mbox{\tiny\boldmath$m$}} + \Sigma_{\mbox{\tiny\boldmath$n$}}}$ (we use here notation from Theorem~\ref{Thm_den} and ignore incomplete cubes, i.e.\ the ones that intersect the ``skew" faces of $E^{(\sbm; \sbn )}_0$, as their contribution to the integrals will be asymptotically negligible as~$q\to \infty$). Fixing one of these cubes, we observe that the probability of the arrival of the   chain $\kk{Z}$ starting at the point $\bx$ to the ``left bottom" vertex of the cube after $\lfloor q t\rfloor$ steps will be given by $h^{\Sigma_{\mbox{\tiny\boldmath$m$}} + \Sigma_{\mbox{\tiny\boldmath$n$}}}$ times a product approximating the quantity $p(\bx, \bz)$ similar to $p(\bv^*, \bz)$ from~\eqref{dens}  (recall that we are dealing with Model~I here, so that we do not need the ``extended" state variable $\bz^*$). Thus $\kk{T}^{\lfloor q t\rfloor} f (\bx)$ will essentially be an integral sum approximating the integral $T(t) f(\bx)$, and as the function $f$ is continuous, it is a simple technical exercise to show that~\eqref{conva} holds true.

The last step in the proof is to observe that a bound of the form~\eqref{prBound_1} will hold uniformly in $q,n\in\N$ for the processes $\{{Z}^{\langle n\rangle} (t) \}_{t\ge 0}$ and $\{\kk{Z}^{\langle n\rangle} (\lfloor qt \rfloor)\}_{t\ge 0}$ as well (resurrecting now the superscripts $\langle n\rangle$). Therefore there exists a $t_\varepsilon<\infty$ such that (recall that we assumed zero initial conditions for all the processes $Z^{\langle n\rangle}$ and $\kk{Z}^{\langle n\rangle}$) one has
\[
  \sup_{q>0} \sup_{B\in \mathscr{B}(\R^K) }
   \Bigl[
   \bigl|\pr (Z^{\langle n\rangle} (t_\varepsilon)\in B )-\pi^{\langle n\rangle}(B) \bigr|
   +
   |\pr (\kk{Z}^{\langle n\rangle}( qt_\varepsilon)\in B )-\kk{\pi}^{\langle n\rangle}(B) \bigr|
   \Bigr]
   <\varepsilon.
  \]
Now the desired assertion follows from~\eqref{pipi} and the weak convergence of the distributions of $\kk{Z}^{\langle n\rangle}( qt_\varepsilon)$ to that of $Z^{\langle n\rangle} (t_\varepsilon)$ as $q\to\infty.$ Theorem~\ref{Thm_disc} is proved.
\end{proof}

\end{document}